\documentclass[a4paper,12pt]{amsart}
\usepackage[english]{babel}
\usepackage{amsmath,amsthm,amssymb,amsfonts}
\usepackage{multicol}
\usepackage{todonotes}
 \usepackage{enumitem} 
\usepackage[bookmarks=true,hyperindex,pdftex,colorlinks,citecolor=blue]{hyperref}
\hypersetup{colorlinks=true,  linkcolor=blue, citecolor=red, urlcolor=cyan}

\usepackage{xcolor}
\usepackage[export]{adjustbox}
\usepackage{graphicx}
\usepackage{subcaption}

\usepackage{svg}

\usepackage[english]{babel}

\usepackage{graphicx}

\usepackage{tikz}
\usetikzlibrary{mindmap,backgrounds, calc}
\usetikzlibrary{matrix,chains,scopes,positioning,arrows,fit}
\usetikzlibrary{positioning, shapes, patterns}
\usetikzlibrary{automata} 
\usetikzlibrary{shapes.geometric, arrows, arrows.meta}
\usetikzlibrary{positioning,decorations.markings}

\theoremstyle{definition}
\newcounter{maincoro}

\newtheorem{theorem}{Theorem}[section]
\newtheorem{lemma}[theorem]{Lemma}
\newtheorem{fact}[theorem]{Fact}

\newtheorem{proposition}[theorem]{Proposition}
\newtheorem{corollary}[theorem]{Corollary}

\theoremstyle{definition}
\newcounter{maintheorem}

\theoremstyle{remark}

\numberwithin{equation}{section}

\newcommand{\R}{\mathbb{R}}

\newcommand{\N}{\mathbb{N}}

 \makeatletter
\renewcommand{\tocsection}[3]{%
	\indentlabel{\@ifnotempty{#2}{\bfseries\ignorespaces#1 #2\quad}}\bfseries#3}
\renewcommand{\tocsubsection}[3]{%
	\indentlabel{\@ifnotempty{#2}{\ignorespaces#1 #2\quad}}#3}

\newcommand\@dotsep{4.5}
\def\@tocline#1#2#3#4#5#6#7{\relax
	\ifnum #1>\c@tocdepth 
	\else
	\par \addpenalty\@secpenalty\addvspace{#2}%
	\begingroup \hyphenpenalty\@M
	\@ifempty{#4}{%
		\@tempdima\csname r@tocindent\number#1\endcsname\relax
	}{%
		\@tempdima#4\relax
	}%
	\parindent\z@ \leftskip#3\relax \advance\leftskip\@tempdima\relax
	\rightskip\@pnumwidth plus1em \parfillskip-\@pnumwidth
	#5\leavevmode\hskip-\@tempdima{#6}\nobreak
	\leaders\hbox{$\m@th\mkern \@dotsep mu\hbox{.}\mkern \@dotsep mu$}\hfill
	\nobreak
	\hbox to\@pnumwidth{\@tocpagenum{\ifnum#1=1\bfseries\fi#7}}\par
	\nobreak
	\endgroup
	\fi}
\AtBeginDocument{%
	\expandafter\renewcommand\csname r@tocindent0\endcsname{0pt}
}
\def\l@subsection{\@tocline{2}{0pt}{2.5pc}{5pc}{}}
\makeatother

\DeclareMathOperator{\co}{co}

\newcommand{\nn}[1]{{\left\vert\kern-0.25ex\left\vert\kern-0.25ex\left\vert #1 
		\right\vert\kern-0.25ex\right\vert\kern-0.25ex\right\vert}}

\newcommand{\ep}{\varepsilon}

\newcommand\restr[2]{\ensuremath{\left.#1\right|_{#2}}}

\thanks{}

\subjclass[2020]{}

\date{\today}

\keywords{}

\begin{document}

\title[Lipschitz approximability]{On the constant of Lipschitz approximability}

\author[R. Medina]{Rubén Medina}
\address[R. Medina]{Universidad de Granada, Facultad de Ciencias. Departamento de Análisis Matemático, 18071-Granada (Spain); and Czech Technical University in Prague, Faculty of Electrical Engineering. Department of Mathematics, Technická 2, 166 27 Praha 6 (Czech Republic) \newline
\href{https://orcid.org/0000-0002-4925-0057}{ORCID: \texttt{0000-0002-4925-0057}}}
\email{rubenmedina@ugr.es}


\thanks{}

\date{\today}
\keywords{Lipschitz, retractions, approximation properties}
\subjclass[2020]{46B20, 46B80, 51F30, 54C15}

\begin{abstract} 
In this note we find $\lambda>1$ and give an explicit construction of a separable Banach space $X$ such that there is no $\lambda$-Lipschitz retraction from $X$ onto any compact convex subset of $X$ whose closed linear span is $X$. This is closely related to a well-known open problem raised by Godefroy and Ozawa in 2014 and represents the first known example of a Banach space with such a property.\end{abstract}
\maketitle
\tableofcontents

\section{Introduction}

\subsection{Motivation and background}
This paper is motivated by a natural question asked by Godefroy and Ozawa in \cite{GO14}, and then subsequently in \cite{God15}, \cite{God215}, \cite{GMZ16}, \cite{GP19} and \cite{God20}.  They wonder whether for every separable Banach space there is a Lipschitz retraction onto a compact convex subset whose closed linear span is the whole space. Recall that this question only makes sense for separable spaces since any space with such a property must be generated by a compact set. In \cite[Theorem 3.3]{HM21}, the latter question is solved in the positive for Banach spaces with a finite dimensional decomposition. In fact, for every $\ep>0$ the retraction can be constructed to be $(1+\ep)$-Lipschitz whenever the space has a monotone Schauder basis. However, the general problem is still open although a Hölder version of the problem was solved positively in \cite[Theorem 2.1]{Med23}.

We prove that there exists a separable Banach space $X$ such that there is no $\lambda$-Lipschitz retraction onto any compact convex subset of $X$ whose closed linear span is the whole space for some $\lambda>1$. However, our example does have such a retraction with larger Lipschitz constant (see Corollary \ref{corl1}).

The Godefroy-Ozawa question is closely related to the theory of nonlinear approximation properties. Indeed, if there is a $\lambda$-Lipschitz retraction $R$ from  a Banach space $X$ onto a generating compact convex subset $K$ of $X$ then there is a sequence of $\lambda$-Lipschitz retractions $R_n:X\to X$ with compact range and pointwise converging to the identity on $X$. Specifically, if we fix $k_0\in K$, the mapping $R_n$ may be defined as $R_n(x)=nR\big(\frac{x}{n}+k_0\big)-nk_0$ which is a translation of $R$ (so that $0$ is in the image) composed with an expansion. Notice that $R_n(x)\to x$ for all $x\in X$ since $\overline{\bigcup_nR_n(X)}=\overline{\bigcup_nn(K-k_0)}=X$. Therefore, Godefroy-Ozawa question is strongly related to Kalton's Problem 1 in \cite[page 1260]{Kal12}, namely, whether in every separable Banach space $X$ there is a sequence of equi-uniformly continuous mappings with compact range pointwise converging to the indentity on $X$. It is worth mentioning that a counterexample to Kalton's problem would provide a renorming of $\ell_1$ without the metric approximation property, solving a classical open problem in the theory of approximation properties.

Nonlinear approximation properties are intimately connected to the linear ones (for instance, see \cite[Theorem 5.3]{GK03}, \cite[Corollary 4.8]{HM21} and \cite[Corollary 2.9]{HM22}). For this exact reason, in our desire of solving Godefroy-Ozawa question, we have decided to follow the path traced by Enflo to construct a separable Banach space without a basis (\cite{Enf73}). He first tackled the problem of finding a Banach space with basis constant greater than 1 (see \cite{Enf73b}) and so have we done in our analogous nonlinear setting. Moreover, Enflo himself mentions in \cite[Second paragraph of page 309]{Enf73b} that the approach developed in \cite{Enf73} is similar to that of \cite{Enf73b}, which we find highly encouraging.

Section \ref{mainsec} is divided into two subsections. In order to prevent the reader from losing the intuition among the computations developed in our proof, we have added a short subsection explaining the general ideas behind the construction. In the last subsection, both the construction and the proof of our main result Theorem \ref{mainTH} are carried out.

\subsection{Definitions and notation.}

Let $(M,d_M)$ and $(N,d_N)$ be two metric spaces and let $f:M\to N$ be an arbitrary mapping. We will say that $f$ is \textit{Lipschitz} if there is $\lambda\ge0$ such that $d_N(f(x),f(y))\le\lambda d_M(x,y)$ for every $x,y\in M$. We may specify that some constant $\lambda\ge0$ satisfy the latter inequality saying that $f$ is $\lambda$-Lipschitz. The infimum over all $\lambda>0$ such that $f$ is $\lambda$-Lipschitz is called the \textit{Lipschitz constant} of $f$.

A \textit{retraction} from a metric space $(M,d)$ onto a subset $N\subset M$ is a map $R:M\to N$ satisfying that $R(x)=x$ for every $x\in N$. The image of a retraction is called a \textit{retract}. We say that $R$ is a \textit{Lipschitz retraction} or equivalently that $R(M)$ is a \textit{Lipschitz retract} whenever the retraction $R$ is Lipschitz.

If $X$ is a Banach space, we say that a subset $S$ of $X$ is a \textit{generating} subset whenever the closed linear span of $S$ is $X$. We denote the closed linear span of $S$ as $\overline{\text{span}}(S)$.

Given some countable set $\Gamma$, we denote $\ell_\infty(\Gamma)$ the space of bounded functions $x:\Gamma\to\R$ endowed with the supremum norm $\|x\|_\infty=\sup\limits_{\gamma\in\Gamma}|x(\gamma)|$. Whenever we are working with an element $x$ belonging to some $\ell_\infty(\Gamma)$ we will denote the supremum norm of $x$ simply by $\|x\|$. We also denote $\ell_1(\Gamma)$ the space of summable functions $x:\Gamma\to\R$ endowed with the norm $\|x\|_1=\sum\limits_{\gamma\in\Gamma}|x(\gamma)|$. For $N\in\N$ we will refer as $\ell_\infty^N$ and $\ell_1^N$ to the space $\ell_\infty(\Gamma)$ and $\ell_1(\Gamma)$ respectively for $\Gamma=\{1,\dots,N\}$. We say that a sequence $(e_n)_{n\in\N}$ of normalised vectors from a Banach space $E$ is \textit{equivalent to the} $\ell_1(\N)$ \textit{basis} if for every sequence $(\lambda_n)_{n\in\N}\subset \R$ with finitely many nonzero elements it holds that
$$\bigg\|\sum\limits_{n\in\N}\lambda_ne_n\bigg\|=\sum_{n\in\N}|\lambda_n|.$$
It is worth mentioning that the given definition is not the usual concept of basic sequences being equivalent.

Given a Banach space $E$, two nonempty subsets $\mathcal{S}_1$ and $\mathcal{S}_2$ of $E$ and an element $v\in E$, we say that $v$ is a \textit{midpoint} of $\mathcal{S}_1$ and $\mathcal{S}_2$ whenever for every $z\in\mathcal{S}_1\cup\mathcal{S}_2$, $\|v-z\|=d(\mathcal{S}_1,\mathcal{S}_2)/2.$

Throughout the entire note, we will consider $\N$ the natural numbers starting from $1$ and $\N_0=\N\cup\{0\}$.

Throughout the entire note we will follow the terminology and notation used in \cite{Fab1}. For the background on Lipschitz and uniformly continuous retractions we refer the reader to the first two chapters of the authoritative monograph \cite{BL2000}.

\section{Main result}\label{mainsec}
\subsection{Sketch of the proof}
Let us begin with a somehow imprecise but useful formulation of the approach used in subsection \ref{sub2}. We will try to construct $X$ such that if $R$ is a $(1+\ep)$-Lipschitz retraction with convex image $K$ (where $0\in K$) for some small enough $\ep>0$ almost preserving some fixed vectors $\pm U_1,\pm U_2,\pm V_{1,1}, \pm V_{2,1}\in X$ then there must be a sequence in $K$ without any Cauchy subsequence. The reasoning is going to follow an iterative argument. More precisely, out of the eight vectors $\pm U_1,\pm U_2,\pm V_{1,1}, \pm V_{2,1}$ we will be able to find another four vectors $\pm V_{1,2},\pm V_{2,2}$ distant from the rest which are also almost preserved by $R$. Then, using the fact that $\pm U_1,\pm U_2,\pm V_{1,2}, \pm V_{2,2}$ are almost preserved by $K$ we will find another four vectors $\pm V_{1,3},\pm V_{2,3}$ distant from the rest which are again almost preserved by $R$ and so on. The sought sequence without Cauchy subsequences is going to be $(R(V_{1,n}))_{n\in\N}$. The unique problem being how to produce the new four almost preserved vectors $\pm V_{1,n+1},\pm V_{2,n+1}$ out of the previous almost preserved vectors $\pm U_1,\pm U_2,(\pm V_{1,m})_{m=1}^n, (\pm V_{2,m})_{m=1}^n$:

For technical reasons, the argument depends on whether $n$ is even or odd. Assume first that $n$ is odd and $\pm U_1,\pm U_2, \pm V_{1,n},\pm V_{2,n}$ are almost preserved by $R$. We need to consider two subsets of vectors $\mathcal{S}_n^+,\mathcal{S}^-_n\subset X$ such that \\$\mathcal{S}_n^{\pm}\subset\co\{R(\pm U_1), R(\pm U_2), R(\pm V_{1,n}), R(\pm V_{2,n})\}\subset K$. More precisely, the sets $\mathcal{S}_n^\pm$ contain 8 vectors $\mathcal{S}_n^\pm=\{Z^\pm_{1,n},\dots,Z^\pm_{8,n}\}$ chosen such that the elements $Z^+_{1,n},\dots,Z^+_{8,n}$ are close to $\frac{U_1+\frac{V_{1,n}+V_{2,n}}{2}}{2}$ and the elements $Z^-_{1,n},\dots,Z^-_{8,n}$ are close to $-\frac{U_1+\frac{V_{1,n}+V_{2,n}}{2}}{2}$. The space $X$ is constructed so that $ V_{1,n+1}$ is the unique midpoint of $\{0\}$ and $\mathcal{S}_n^+\setminus\{Z^+_{2,n}\}$ and the vector $V_{2,n+1}$ is the unique midpoint of $\{0\}$ and $\mathcal{S}_n^-\setminus\{Z^-_{3,n}\}$. Since all the latter sets are in $K$ their elements are preserved by $R$ and since there are no other midpoints of those sets and \{0\} rather than $V_{1,n+1},V_{2,n+1}$, these vectors must also be almost preserved by the $(1+\ep)$-Lipschitz retraction $R$.

We may define $U_1,U_2,V_{1,n},V_{2,n}$ satisfying $\|U_1\|,\|U_2\|=1$, $\|V_{1,n}\|,\|V_{2,n}\|\approx 1/3$ and 

$$d(\{0\},\mathcal{S}^\pm_{n})\approx\frac{\|U_1\|+\frac{\|V_{1,n}\|+\|V_{2,n}\|}{2}}{2}\approx 2/3\;\;\;\text{ for }i=2,3.$$
Hence, the vectors $V_{1,n+1},V_{2,n+1}$ must satisfy that $\|V_{1,n+1}\|,\|V_{2,n+1}\|=d(\{0\},\mathcal{S}^\pm_{n})/2\approx 1/3\approx \|V_{1,n}\|,\|V_{2,n}\|$. This is the key ingredient which prevents the sequence $(V_{1,n})_{n\in\N}$ from converging to $0$. Therefore, it is possible to define the vectors distant from each other. If, on the contrary, $n$ is even, we proceed analogously but using $U_2$ instead of $U_1$.

Let us now introduce the purpose that our main lemmas
below serve. Lemma \ref{prelim2} shows that $\pm V_{i,n+1}$ is the midpoint of $\{0\}$ and its respective set $\mathcal{S}$ described above whereas Lemma \ref{prelim3} proves that $\pm V_{i,n+1}$ is the unique element with such a property. Moreover, since the argument must work up to $\ep>0$, we show in Lemma \ref{prelim3} that if a vector $B$ is almost the midpoint (up to $\ep$) of $\{0\}$ and $\mathcal{S}$ then $B$ is almost $\pm V_{i,n+1}$ (up to a constant $C>0$ times $\ep$). Finally, Lemmas \ref{prelim2} and \ref{prelim3} are used in the proof of Theorem \ref{precth} (a more precise formulation of Theorem \ref{mainTH}) to finish the section.

\subsection{Final construction}\label{sub2}
We will give an explicit definition of a separable Banach space $X$ and a proof of the fact that there is no $\lambda$-Lipschitz retraction onto any generating compact and convex subset of $X$ for some $\lambda>1$. That is, we will be proving the following result.

\begin{theorem}\label{mainTH}
    There is $\lambda>1$ and a separable Banach space $X$ such that no retraction onto any generating compact and convex subset of $X$ is $\lambda$-Lipschitz.
\end{theorem}

The space $X$ will arise as the closed linear span of countably many carefully chosen vectors from $\ell_\infty(\N^2)$.

Let us then proceed with the definition of $X$. The construction will depend on some constants $\delta,M,\Delta>0$ which will be specified later. We set for every $1\le k\le8$ the following quantities,
$$\begin{aligned}x_k=\frac{\frac{1}{2}+\frac{k}{M}}{1-\frac{2k}{M}}\;\;&,\;\;t(k)=\frac{1}{2}-\frac{k}{M}\\
\alpha=1-x_8\;\;&,\;\;\beta=1+x_1^2,\end{aligned}$$
We now start defining functions $f,g:\R\to\R$ given by 
$$f(x)=\alpha+x\;\;\text{ and }\;\;g(x)=\beta-x^2\;\;\;\;\;(x\in\R).$$
We also define three vectors $u,v,c\in \R^8$ given by $u(k)=f(x_k)$, $v(k)=g(x_k)$ and $c(k)=1$ for $k=1,\dots,8$.
\begin{fact}\label{factou1} The vector $(x_k)_{k=1}^8$ is strictly increasing in $k=1,\dots,8$. Moreover, if $M>16$ and we set $\delta_M:=x_8^2-x_1^2$ then $\delta_M>x_8-x_1>0$, $\delta_M\xrightarrow{M\to\infty}0$ and
    $$\begin{aligned}\max\limits_{k=1,\dots,8}u(k)=u(8)=1\;\;\;&, \;\;\;\min\limits_{k=1,\dots,8}u(k)=u(1)=1-(x_8-x_1),\\ \max\limits_{k=1,\dots,8}v(k)=v(1)=1\;\;\;&,\;\;\;
    \min\limits_{k=1,\dots,8}v(k)=v(8)=1-(x_8^2-x_1^2).\end{aligned}$$
\end{fact}
\begin{proof}
    $M>16$ implies $x_k>\frac{1}{2}$ for every $1\le k\le8$. Then, $x_8+x_1>1$ and hence
    $$\delta_M=(x_8+x_1)(x_8-x_1)>x_8-x_1.$$
    The rest of the proof is straightforward.
\end{proof}

From now on, we will consider a sequence $(e_n)\subset S_{\ell_\infty(\N)}$ equivalent to the $\ell_1(\N)$ basis in $\ell_\infty(\N)$. This is possible since $\ell_\infty(\N)$ is isometrically universal for separable spaces. Also, let us consider  $\delta_i\in\R^8$ for $i=1,\dots,8$ given by $\delta_i(k)=0$ if $i\neq k$ and $\delta_i(i)=\Delta$.

We are finally ready to present the vectors from $\ell_\infty(\N^2)$ that will span our space. All the vectors will be defined following the same approach, namely, each vector will be defined in $\{m\}\times\N$ in a specific way for each $m\in\N$. In the case $m=1$, the definition will be splitted into blocks of 8 elements in a row, that is, the first block will be formed by $(1,1),\dots,(1,8)$, the second will be formed by $(1,8+1),\dots,(1,2\cdot 8)$ and the $n^\text{th}$ block will be formed by $(1,(n-1)8+1),\dots,(1,n8)$ (see Figure \ref{PICT1}). We denote as \textit{BLOCK m} the elements $(1,m8+1),\dots,(1,(m+1)8)$ for every $m\in\N_0$.

\begin{figure}
    \centering
    \includegraphics[scale=0.8]{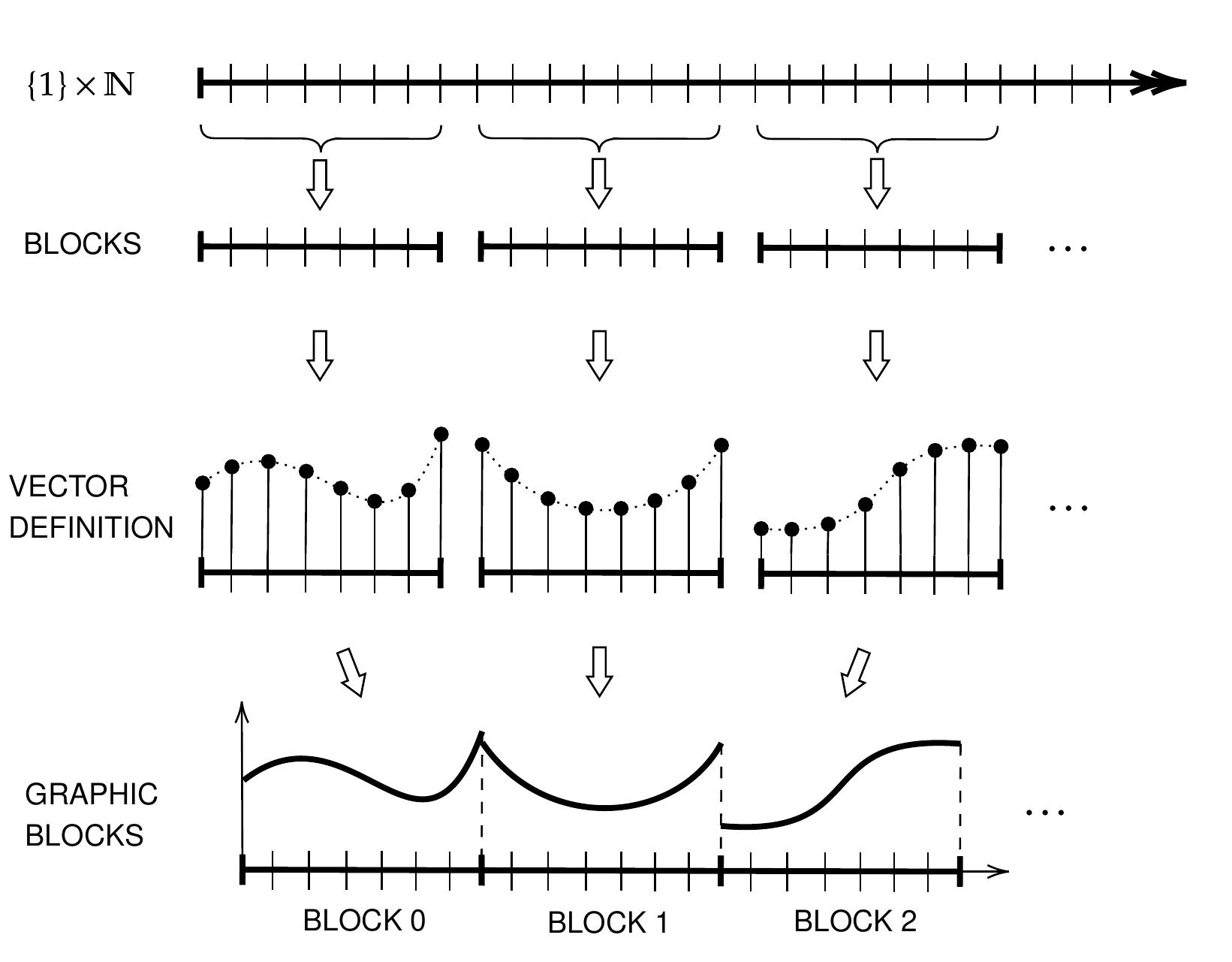}
    \caption{}
    \label{PICT1}
\end{figure}

Consider now the elements $U_1,U_2\in\ell_\infty(\N^2)$ given by
$$U_i(x)=\begin{cases}u(k)\;\;\;&\text{ if }x=(1,8(2m+j-1)+k)\text{ for }m\in\N_0,\;k=1,\dots,8,\\
\delta_8(k)&\text{ if }x=(1,8(2m-j)+k)\text{ for }m\in\N,\;k=1,\dots,8,\\\frac{1}{10}e_1(k)&\text{ if }x=(p,k)\text{ for }p,k\in\N,\;p\ge2.\end{cases}$$


Then, we define for every $n\in \N$ the elements $V_{1,n},V_{2,n}\in\ell_\infty(\N^2)$ as
$$V_{i,n}(x)=\begin{cases}
(-1)^{i-1}(\frac{1-\delta}{3}-\delta_{i+1}(k))&\text{ if }x=(1,n8+k)\text{, }k=1,\dots,8\\
v(k)-\frac{2+2\delta}{3}+(-1)^i\delta_{7}(k)&\text{ if }x=(1,(n+1)8+k)\text{, }k=1,\dots,8\\
(-1)^{i-1}(\frac{1-2\delta}{3}-\delta_{6}(k))&\text{ if }x=(1,m8+k)\text{, }k=1,\dots,8,\;m\in2\N_0+\frac{1-(-1)^n}{2}\setminus\{n\},\\
(-1)^{i-1}\frac{1-\delta}{3}e_1(k)&\text{ if }x=(2n+i-1,k)\text{, }k\in\N,\\
\frac{1}{10}e_{2n+i-1}(k)&\text{ if }x=(m,k)\text{, }\left|\begin{aligned}&m,k\in\N,\;m\ge2,\\\;&m\neq 2n+i-1,2(n+1),2(n+1)+1,\end{aligned}\right.\\
0&\text{ elsewhere.}\end{cases}$$

Our sought space is
$$X=\overline{\text{span}}\Big(\{V_{i,n}\}_{\substack{i=1,2\\n\in\N}}\cup\{U_1,U_2\}\Big).$$
See Figure \ref{UI} and Figure \ref{VI} to have a conceptual idea of what the vectors $U_i,V_{i,n}$ look like when restricted to $\{1\}\times\N$ (we present the picture by giving the graph of the blocks arranged like in the last step of Figure \ref{PICT1}). Notice that, when restricted to $\{m\}\times\N$ for $m\neq1$, the previous vectors are nothing but multiples of some vector $e_j$ from the $\ell_1$ basis. We have included in Figure \ref{VI} a description of $\restr{\frac{V_{1,n}+V_{2,n}}{2}}{\{1\}\times\N}$ because it is also going to be central in further discussions.

\begin{figure}
    \centering
    \includegraphics[scale=0.5]{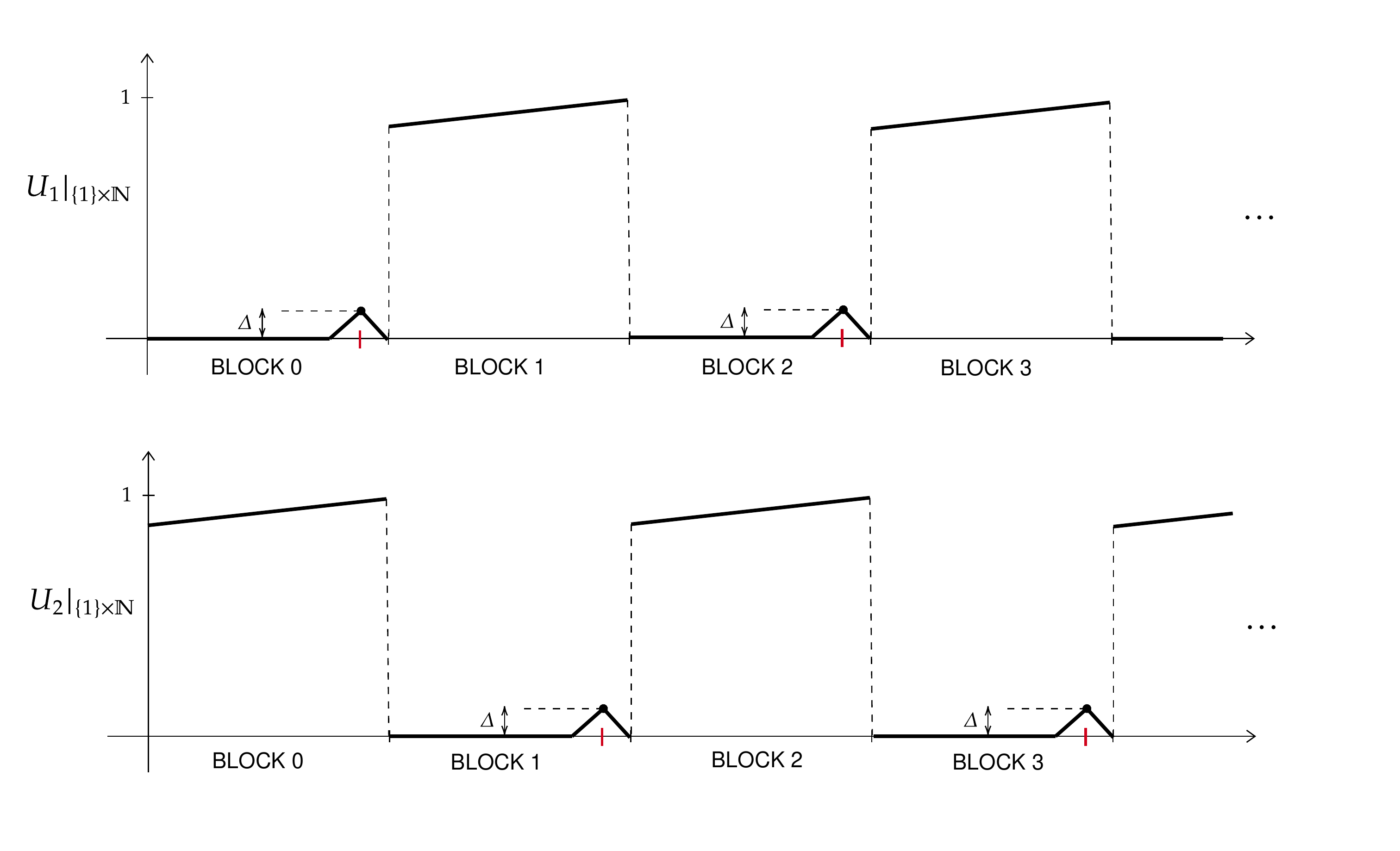}
    \caption{}
    \label{UI}
\end{figure}

\begin{figure}
    \centering
    \includegraphics[scale=0.5]{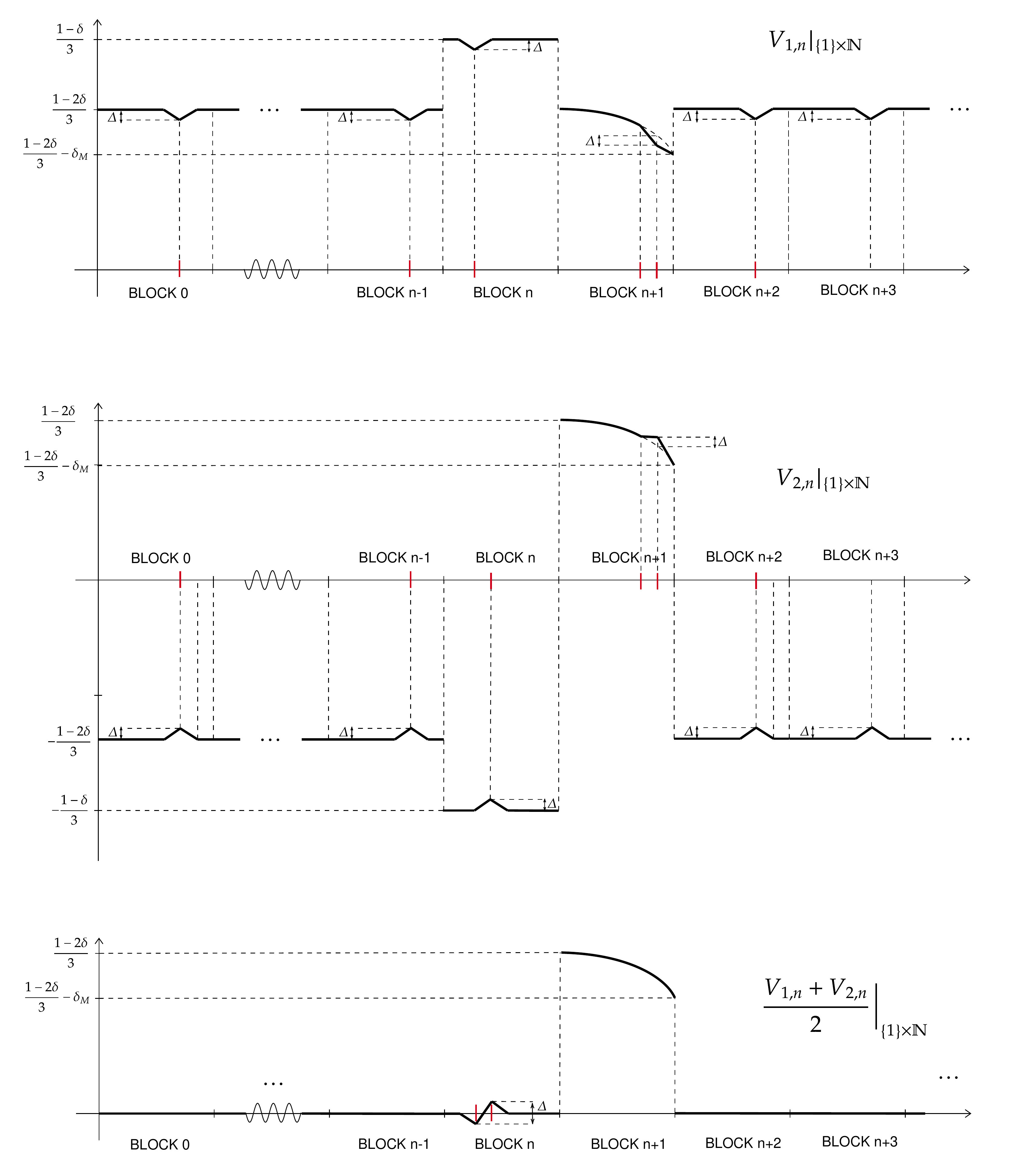}
    \caption{}
    \label{VI}
\end{figure}

Now, we turn our attention to the proof of Theorem \ref{mainTH}. For that purpose, we first need to state and prove some technical auxiliary results. We start with the following simple but handy lemma.

\begin{lemma}\label{prelim1}
Let $N,n\in\N$ with $N>n$, $\{v_1,\dots,v_n\}\in \ell_\infty^N\setminus\{0\}$ and $c=(1,\dots,1)\in \ell_\infty^N$. If $v_1,\dots,v_n,c$ are linearly independent then there exists $K>0$ such that for every $a\in \text{span}\{v_1,\dots,v_n,c\}$ with $a=\lambda c+\sum\limits_{i=1}^n\lambda_iv_i$  it follows that
$$|\lambda-1|,|\lambda_i|\le K\|a-c\|\;,\;\;\;\;i=1,\dots,n.$$
\end{lemma}
\begin{proof}
    It follows immediately from the fact that the sequence $\{v_1,\dots,v_n,c\}$ is a block basis but, for the convienence of the reader, we include here a complete proof.
    
    Since $v_1,\dots,v_n,c$ are linearly independent, we know that there exist $f_0,f_1,\dots,f_n\in(\ell_\infty^N)^*=\ell_1^N$ such  that $z=f_0(z)c+\sum\limits_{i=1}^nf_i(z)v_i$ for every $z\in\ell_\infty^N$. Moreover, we know that $f_0(c)=1$ and $f_i(c)=0$ for $i=1,\dots,n$. Let us consider $K=\max\{\|f_0\|,\|f_1\|,\dots,\|f_n\|\}<\infty$. It is clear that for every $1\le i\le n$,
    $$|\lambda_i|=|f_i(a)|=|f_i(a-c)|\le\|f_i\|\cdot\|a-c\|\le K\|a-c\|.$$
    Finally, it is immediate that
    $$|\lambda-1|=\|f_0(a-c)\|\le\|f_0\|\|a-c\|\le K\|a-c\|.$$
\end{proof}


We will make use of the following quantity which depends exclusively on $M>0$,
\begin{equation}\label{eqmu}\mu_M=\min\limits_{\substack{j,k=1,\dots,8\\j\neq k}}\{(1-t(k))(u(k)-u(j))+t(k)(v(k)-v(j))\}.\end{equation}
\begin{fact}\label{factou2}
For every $M>16$ it holds that $\mu_M>0$ and $\mu_M\xrightarrow{M\to\infty}0$.
\end{fact}
\begin{proof}
    Let us consider for $1\le k\le8$ the function $h_k:\R\to\R$ given by
    $$h_k(x)=(1-t(k))f(x)+t(k)g(x)\;\;\;\;\;(x\in \R).$$
    It is immediate that
    $$\mu_M=\min\limits_{\substack{j,k=1,\dots,8\\j\neq k}}\{h_k(x_k)-h_k(x_j)\}.$$
    For every choice of $M>16$, the function $h_k$ is a quadratic polynomial attaining a unique maximum at $x_k$. Therefore, if $j\neq k$ then $h_k(x_k)>h_k(x_j)$ and thus $\mu_M>0$. Now, since $h_k$ is contiuous and $x_k,x_j\xrightarrow{M\to\infty}1/2$, we clearly have
    $$h_k(x_k)-h_k(x_j)\xrightarrow{M\to\infty}h_k\Big(\frac{1}{2}\Big)-h_k\Big(\frac{1}{2}\Big)=0$$
    and we are done.    
\end{proof}

 Let us finally pick appropriate $\delta,M,\Delta>0$ for which $X$ will satisfy the statement of Theorem \ref{mainTH} for some $\lambda>1$.\\

\textbf{Choice of }$\delta,M$ \textbf{and} $\Delta$:\\
We take some $\delta>0$ satisfying
\begin{equation}\label{con0}
    \delta<\frac{1}{10},
\end{equation} 
and find $M$ large enough so that the following inequalities hold,\\
\begin{minipage}{0.4\textwidth}
    \begin{equation}\label{con1}
        \frac{8}{M}+\delta_M+\mu_M<\frac{1-8\delta}{6},
    \end{equation}
\end{minipage}
\begin{minipage}{0.2\textwidth}
\,
\end{minipage}
\begin{minipage}{0.4\textwidth}
    \begin{equation}\label{con2}
        \delta_M<\frac{\delta}{3}.
    \end{equation}
\end{minipage}

Finally, we pick $\Delta>0$ so that
\begin{equation}\label{con3}\Delta<\frac{\mu_M}{2},\;\frac{\delta}{6},\;\frac{1-2\delta}{3}-\delta_M,\delta_M,\frac{1-\delta_M}{2}.\end{equation}

We denote now for $n\in\N$ and $1\le k\le8$ the element $S_{k,n}\in X$ given by
    $$S_{k,n}=(1-t(k))U_{\frac{3+(-1)^{n}}{2}}+t(k)\frac{V_{1,n}+V_{2,n}}{2}.$$

The element $S_{k,n}\in X$ will be very relevant for the proof of Theorem \ref{mainTH}. See Figure \ref{SK} to have a rough idea of what it looks like when restricted to $\{1\}\times\N$.

\begin{figure}
    \centering
    \includegraphics[scale=0.5]{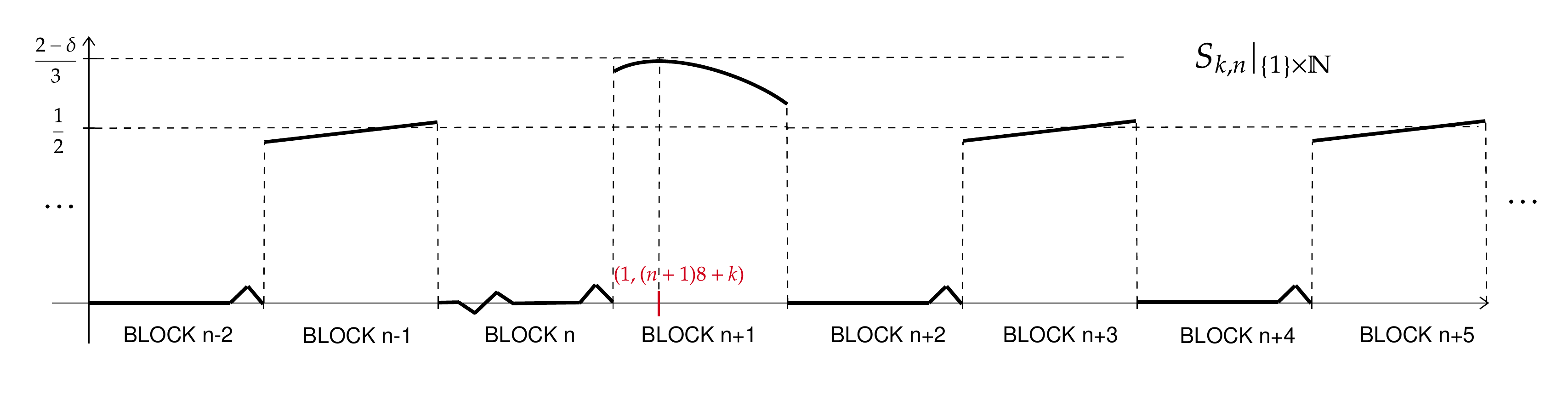}
    \caption{}
    \label{SK}
\end{figure}

For the latter choices of $\delta,M$ and $\Delta$, we have the following proposition.

\begin{proposition}\label{c1}
The following properties are satisfied for every $n\in\N$,
\begin{enumerate}
    \item\label{p2} $\|V_{i,n}\|=\frac{1-\delta}{3}$ and $\|U_i\|=1$ for $i=1,2$.
    \item \label{p5} $|(V_{1,n}+V_{2,n})(1,m)|\le\Delta$ for every $m\in\N\setminus\{(n+1)8+1,\dots,(n+1)8+8\}$.
    \item\label{p2ymedio} For every $k=1,\dots,8$,
    $$S_{k,n}(1,(n+1)8+k)> \frac{2-\delta}{3}-\delta_M+\frac{\Delta}{2}.$$
    \item\label{p3} For every $x\in\N^2\setminus\{(1,(n+1)8+k)\}$,
    $$|S_{k,n}(x)|\le S_{k,n}(1,(n+1)8+k)-\mu_M.$$
\end{enumerate}
\end{proposition}
\begin{proof}\;\\

    \eqref{p2}: It is clear that $\|U_i\|=u(8)=1$. Now, taking into account the constraints of $\delta,M$ and $\Delta$, a straightforward case by case check yields that $\|V_{i,n}\|=\frac{1-\delta}{3}$.\\

    \eqref{p5}: The proof is straightforward from the definition.\\

    \eqref{p2ymedio}: By Fact \ref{factou1},
\begin{equation}\label{eq1}
\begin{aligned}
    u(1)+v(1)=2-(x_8-x_1)>2-\delta_M.
    \end{aligned}
\end{equation}
We also deduce that 
\begin{equation}\label{eq307}u(1)=1-(x_8-x_1)\stackrel{\text{Fact }\ref{factou1}}{>}1-\delta_M\stackrel{\eqref{con2}}{>}\frac{1-2\delta}{3}=v(1)-\frac{2+2\delta}{3}.\end{equation}
This, in addition to the fact that $1-t(k)>t(k)$,
yields
\begin{equation}\label{eq2}(1-t(k))u(1)+t(k)\Big(v(1)-\frac{2+2\delta}{3}\Big)\stackrel{\eqref{eq307}}{>}\frac{u(1)+v(1)-\frac{2+2\delta}{3}}{2}.\end{equation}
Finally, since $\mu_M>0$ we know that $(1-t(k))u(k)+t(k)(v(k))\ge(1-t(k))u(1)+t(k)v(1)$ and therefore
$$\begin{aligned}S_{k,n}(1,(n+1)8+k)&=(1-t(k))u(k)+t(k)\Big(v(k)-\frac{2+2\delta}{3}\Big)\\
&\;\ge(1-t(k))u(1)+t(k)\Big(v(1)-\frac{2+2\delta}{3}\Big)\\&\stackrel{\eqref{eq2}}{>}\frac{u(1)+v(1)-\frac{2+2\delta}{3}}{2}\\&\stackrel{\eqref{eq1}}{>}\frac{2-\delta_M-\frac{2+2\delta}{3}}{2}\\&\;=\frac{2-\delta}{3}-\frac{\delta_M}{2}\stackrel{\delta_M>\Delta}{>}\frac{2-\delta}{3}-\delta_M+\frac{\Delta}{2}.\end{aligned}$$

\eqref{p3}: We distinguish here between different cases depending on the element $x\in\N^2\setminus\{(1,(n+1)8+k)\}$:\\
\textbf{Case 1, }$x=(1,(n+1)8+j)$ for $j=1,\dots,8$ and $j\neq k$. In this case we have
$$\begin{aligned}|S_{k,n}(x)|=&(1-t(k))u(j)+t(k)\Big(v(j)-\frac{2+2\delta}{3}\Big)\\=&(1-t(k))u(k)+t(k)\Big(v(k)-\frac{2+2\delta}{3}\Big)\\&-\big((1-t(k))(u(k)-u(j))+t(k)(v(k)-v(j))\big)\\\stackrel{\eqref{eqmu}}{\le}&S_{k,n}(1,(n+1)8+k)-\mu_M.\end{aligned}$$
\textbf{Case 2,} $x=(1,q)$ with $q\neq (n+1)8+j$ for $j=1,\dots,8$. From property \eqref{p5} we know that $\big|\frac{V_{1,n}+V_{2,n}}{2}(x)\big|\le\Delta/2$ and thus
$$\begin{aligned}|S_{k,n}(x)|\le&(1-t(k))\big|U_{\frac{3+(-1)^n}{2}}(x)\big|+\Delta/2\\\stackrel{\eqref{p2}}{\le}&1-t(k)+\Delta/2\stackrel{\eqref{con1}}{\le} \frac{2-\delta}{3}-\delta_M-\mu_M+\Delta/2\\\stackrel{\eqref{p2ymedio}}{<}&S_{k,n}(1,(n+1)8+k)-\mu_M.\end{aligned}$$
\textbf{Case 3,} $x=(p,q)$ for some $p,q\in\N$ with $p\ge 2$. In this case 
$$|U_{\frac{3+(-1)^n}{2}}(x)|\le\frac{1}{10}\;\;\;\text{ and }\;\;\;\frac{1}{3}\stackrel{\eqref{con1}}{\le}\frac{2-\delta}{3}-\delta_M-\mu_M\stackrel{\eqref{p2ymedio}}{<}S_{k,n}(1,(n+1)8+k)-\mu_M,$$
so then
$$\begin{aligned}|S_{k,n}(x)|\le&(1-t(k))\big|U_{\frac{3+(-1)^n}{2}}(x)\big|+t(k)\frac{|V_{1,n}(x)|+|V_{2,n}(x)|}{2}\\\stackrel{\eqref{p2}}{\le}&(1-t(k))\frac{1}{10}+t(k)\frac{1-\delta}{3}<\frac{1}{3}< S_{k,n}(1,(n+1)8+k)-\mu_M.\end{aligned}$$
\end{proof}

The following lemma is crucial for our later purposes. Despite being large and seemingly involved, its proof is elementary and the computations needed are not so hard. However, we are forced to work distinguishing cases due to the nature of the definition of the vectors $U_i,V_{i,n}\in X$.

\begin{lemma}\label{prelim2}
    There exists $\ep_0>0$ such that if $E_1,E_2,W_1, W_2\in X$ satisfy $\|W_i-V_{i,n}\|,\|E_i-U_i\|\le\ep_0$ for $i=1,2$ and some $n\in\N$ then for $k=1,\dots,8$ we have that $\|Z_{k,n}\|\ge\frac{2-2\delta}{3}$ and 
    \begin{equation}\label{eqst}\|V_{i,n+1}\|,\Big\|(-1)^{i+1}V_{i,n+1}-\frac{2-2\delta}{3\|Z_{k,n}\|}Z_{k,n}\Big\|=\frac{1-\delta}{3}\;\;\;\;\forall k\neq i+1,\end{equation}
    where $Z_{k,n}=(1-t(k))E_{\frac{3+(-1)^{n}}{2}}+t(k)\frac{W_1+W_2}{2}.$
\end{lemma}
\begin{proof}
We first choose $\ep_0$ satisfying some inequalities which will be used throughout the proof. Let us take $\ep_0>0$ such that
\begin{equation}\label{con4}\ep_0<\frac{\mu_M-2\Delta}{2},\;\frac{1}{40},\;\frac{\delta}{3}-\delta_M,\;\frac{\delta}{3}-2\Delta,\frac{1-8\delta}{6}-\frac{8}{M},\frac{1-10\delta}{60}.\end{equation}
Notice that thanks to the appropriate choice of $\delta,M$ and $\Delta$, the quantities in the right-hand side are strictly positive and hence such an $\ep_0$ exists.
Let us show that this $\ep_0$ satisfies the statement of Lemma \ref{prelim2}. From Properties \eqref{p2ymedio} and \eqref{p3} of Proposition \ref{c1} we get that for $k=1,\dots,8$,
\begin{equation}\label{normin}
    \|S_{k,n}\|=S_{k,n}(1,(n+1)8+k)>\frac{2-\delta}{3}-\delta_M.
\end{equation}
Hence, for every $k=1,\dots,8$ we have that
\begin{equation}\label{eq309}\begin{aligned}\|Z_{k,n}\|\ge&\|S_{k,n}\|-\ep_0\stackrel{\eqref{normin}}{>}\frac{2-\delta}{3}-\delta_M-\ep_0\\=&\frac{2-2\delta}{3}+(\delta/3-\delta_M-\ep_0)\stackrel{\eqref{con4}}{>}\frac{2-2\delta}{3}.\end{aligned}\end{equation}
We turn our attention to the proof of \eqref{eqst}. We already know from Proposition \ref{c1} that $\|V_{i,n+1}\|=\frac{1-\delta}{3}$ so that we only need to prove the second equality of \eqref{eqst}. It is immediate that
$$\Big\|(-1)^{i+1}V_{i,n+1}-\frac{2-2\delta}{3\|Z_{k,n}\|}Z_{k,n}\Big\|\ge\Big\|\frac{2-2\delta}{3\|Z_{k,n}\|}Z_{k,n}\Big\|-\|V_{i,n+1}\|\stackrel{\eqref{p2}}{=}\frac{1-\delta}{3}.$$
Therefore, it is enough to show that for every $x\in\N^2$ and $k\neq i+1$,
\begin{equation}\label{mainpro}
    \Big|\Big((-1)^{i+1}V_{i,n+1}-\frac{2-2\delta}{3\|Z_{k,n}\|}Z_{k,n}\Big)(x)\Big|\le\frac{1-\delta}{3}.
\end{equation}
We proceed by splitting the proof of inequality \eqref{mainpro} into different cases.

\textbf{Case 1,} $x=(1,(n+1)8+k)$. In this case
$$V_{i,n+1}(x)=(-1)^{i-1}\Big(\frac{1-\delta}{3}-\delta_{i+1}(k)\Big)\stackrel{k\neq i+1}{=}(-1)^{i-1}\frac{1-\delta}{3},$$
and since by \eqref{con4} we know that $2\ep_0<\mu_M$ then for every $z\in\N^2\setminus\{x\}$,
$$\begin{aligned}
    |Z_{k,n}(z)|\le|S_{k,n}(z)|+\ep_0&\stackrel{\eqref{p3}}{\le} S_{k,n}(x)-\mu_M+\ep_0\\&\le Z_{k,n}(x)-\mu_M+2\ep_0\\&<Z_{k,n}(x).
\end{aligned}$$
This shows that 
\begin{equation}\label{Zeq}
    \|Z_{k,n}\|=Z_{k,n}(1,(n+1)8+k).
\end{equation}
Therefore,
$$\Big|\Big((-1)^{i+1}V_{i,n+1}-\frac{2-2\delta}{3\|Z_{k,n}\|}Z_{k,n}\Big)(x)\Big|\stackrel{\eqref{Zeq}}{=}\Big|\frac{1-\delta}{3}-\frac{2-2\delta}{3}\Big|=\frac{1-\delta}{3}.$$

\textbf{Case 2,} $x=(1,(n+1)8+j)$ for $j\in\{1,\dots,8\}\setminus\{k\}$. In this case,
\begin{equation}\label{eq301}V_{i,n+1}(x)=(-1)^{i-1}\Big(\frac{1-\delta}{3}-\delta_{i+1}(j)\Big),\end{equation}
\begin{equation}\label{eq302}
    S_{k,n}(x)=(1-t(k))u(j)+t(k)\Big(v(j)-\frac{2+2\delta}{3}\Big).
\end{equation}
Now, since $\delta_M\stackrel{\eqref{con2}}{<}\frac{\delta}{3}\stackrel{\eqref{con0}}{<}\frac{1}{30}$ it holds that
\begin{equation}\label{eq303}v(j)-\frac{2+2\delta}{3}\stackrel{\text{Fact }\ref{factou1}}{\ge}1-\delta_M-\frac{2+2\delta}{3}>0.\end{equation}
Therefore, taking into account that $1-t(k)\ge1/2$ we get that
\begin{equation}\label{eq304}
\begin{aligned}
Z_{k,n}(x)&\ge S_{k,n}(x)-\ep_0\stackrel{\eqref{eq302}}{=}(1-t(k))u(j)+t(k)\Big(v(j)-\frac{2+2\delta}{3}\Big)-\ep_0\\&\stackrel{\eqref{eq303}}{\ge}\frac{u(j)}{2}-\ep_0\stackrel{\text{Fact }\ref{factou1}}{\ge}\frac{1-\delta_M}{2}-\ep_0>0.
\end{aligned}
\end{equation}
The last inequality holds since $\delta_M<\frac{1}{30}$ and $\ep_0\stackrel{\eqref{con4}}{<}\frac{1}{40}$. Combining previous inequalities we get
\begin{equation}\label{eq305}
    \Big((-1)^{i+1}V_{i,n+1}-\frac{2-2\delta}{3\|Z_{k,n}\|}Z_{k,n}\Big)(x)\stackrel{\eqref{eq301}\eqref{eq304}}{<}\frac{1-\delta}{3}.
\end{equation}
Since 
\begin{equation}\label{eq400}
    \|S_{k,n}\|\le(1-t(k))\|U_{\frac{3+(-1)^n}{2}}\|+t(k)\frac{\|V_{1,n}\|+\|V_{2,n}\|}{2}\stackrel{\eqref{p2}}{<}1,
\end{equation}
and $\ep_0\stackrel{\eqref{con4}}{<}\frac{1-4\delta}{3}$ we have that
\begin{equation}\label{help}\|Z_{k,n}\|\le\|S_{k,n}\|+\ep_0\stackrel{\eqref{eq400}}{<}1+\frac{1-4\delta}{3}=\frac{4-4\delta}{3}.\end{equation}
Also,
\begin{equation}\label{eq300}|Z_{k,n}(x)|\le|S_{k,n}(x)|+\ep_0\stackrel{\eqref{p3}}{\le}\|S_{k,n}\|-\mu_M+\ep_0\le\|Z_{k,n}\|-\mu_M+2\ep_0.\end{equation}
Using the above inequalities we deduce that
\begin{equation}\label{eq306}\begin{aligned}
    \Big(\frac{2-2\delta}{3\|Z_{k,n}\|}Z_{k,n}-(-1)^{i+1}&V_{i,n+1}\Big)(x)\stackrel{\eqref{eq301}}{\le}\frac{2-2\delta}{3\|Z_{k,n}\|}|Z_{k,n}(x)|-|V_{i,n+1}(x)|\\\stackrel{\eqref{eq300}}{\le}&\Big(\frac{2-2\delta}{3}-\frac{2-2\delta}{3\|Z_{k,n}\|}(\mu_M-2\ep_0)\Big)-\Big(\frac{1-\delta}{3}-\Delta\Big)\\\stackrel{\eqref{help}}{<}& \frac{1-\delta}{3}-\frac{1}{2}(\mu_M-2\ep_0-2\Delta)\stackrel{\eqref{con4}}{<}\frac{1-\delta}{3}.
\end{aligned}\end{equation}
This case is hence done by combining \eqref{eq305} and \eqref{eq306}.


\textbf{Case 3,} $x=(1,(n+2)8+j)$ with $1\le j\le 8$. In this case we have
$$V_{i,n+1}(x)=v(j)-\frac{2+2\delta}{3}+(-1)^i\delta_7(j),$$
which taking into account that $v(j)\stackrel{\text{Fact }\ref{factou1}}{\le}1$ implies that
\begin{equation}\label{eq4}|V_{i,n+1}(x)|\le\frac{1-2\delta}{3}+\Delta.\end{equation}
We now compute $|S_{k,n}(x)|$. Since we have that $|(V_{1,n}+V_{2,n})(x)|\stackrel{\eqref{p5}}{\le}\Delta$ and $U_{\frac{3+(-1)^n}{2}}(x)=\delta_8(j)$, it clearly follows that
\begin{equation}\label{eq5}|S_{k,n}(x)|\le (1-t(k))|\delta_8(j)|+t(k)\Delta\le\Delta.\end{equation}
Therefore,
$$\begin{aligned}\Big|\Big((-1)^{i+1}V_{i,n+1}-\frac{2-2\delta}{3\|Z_{k,n}\|}Z_{k,n}\Big)(x)\Big|&\stackrel{\eqref{eq309}}{\le}|V_{i,n+1}(x)|+|Z_{k,n}(x)|\\&\le|V_{i,n+1}(x)|+|S_{k,n}(x)|+\ep_0\\&\stackrel{\eqref{eq4}\eqref{eq5}}{\le}\frac{1-2\delta}{3}+2\Delta+\ep_0\\&=\frac{1-\delta}{3}-\Big(\frac{\delta}{3}-2\Delta-\ep_0\Big)\stackrel{\eqref{con4}}{<}\frac{1-\delta}{3}.\end{aligned}$$

\textbf{Case 4,} $x=(1,m8+j)$ where $m\in 2\N_0+\frac{1-(-1)^n}{2}\setminus\{n+2\}$ and $1\le j\le 8$. In this case by definition we get that
$$V_{i,n+1}(x)=0\;\;,\;\;|(V_{1,n}+V_{2,n})(x)|\stackrel{\eqref{p5}}{\le}\Delta\;\;,\;\;U_{\frac{3+(-1)^n}{2}}(x)=\delta_8(j).$$
Hence, $|S_{k,n}(x)|\le(1-t(k))\Delta+t(k)\Delta=\Delta$ and we conclude
$$\begin{aligned}\Big|\Big((-1)^{i+1}V_{i,n+1}-\frac{2-2\delta}{3\|Z_{k,n}\|}Z_{k,n}\Big)(x)\Big|\stackrel{\eqref{eq309}}{\le}|Z_{k,n}(x)|\le|S_{k,n}(x)|+\ep_0\le\Delta+\ep_0\le\frac{1-\delta}{3}.\end{aligned}$$
The very last inequality follows from the constraints established for $\Delta$ and $\ep_0$ in \eqref{con3} and \eqref{con4}.

\textbf{Case 5,} $x=(1,m8+j)$ with $m$ distinct from the previous cases and $1\le j\le8$, that is, $m$ is any number with the opposite parity to $n$ except $n+1$:
$$m\in 2\N_0+\frac{1-(-1)^{n+1}}{2}\setminus\{n+1\}.$$
In this case,
\begin{equation}\label{eq311}V_{i,n+1}(x)=(-1)^{i-1}\Big(\frac{1-2\delta}{3}-\delta_6(j)\Big)\;\;\;\;\;,\;\;\;\;\;U_{\frac{3+(-1)^n}{2}}(x)=u(j).\end{equation}
Then, since $\delta_M\stackrel{\eqref{con2}}{<}\delta/3<1/30$ and $\Delta\stackrel{\eqref{con3}}{<}\delta/6\stackrel{\eqref{con0}}{<}1/60$ we have that 
\begin{equation}\label{eq310}\begin{aligned}S_{k,n}(x)&\stackrel{\eqref{eq311}}{=}(1-t(k))u(j)+t(k)\frac{(V_{1,n}+V_{2,n})(x)}{2}\\&\stackrel{\eqref{p5}}{\ge}(1-t(k))u(j)-\Delta\\&\stackrel{\text{Fact }\ref{factou1}}{\ge}(1-t(k))(1-\delta_M)-\Delta\\&\ge\frac{1-\delta_M}{2}-\Delta>0.\end{aligned}\end{equation} Then,
\begin{equation}\label{eqa}\begin{aligned}\Big((-1)^{i+1}V_{i,n+1}-\frac{2-2\delta}{3\|Z_{k,n}\|}S_{k,n}\Big)(x)&\stackrel{\eqref{eq310}}{\le}(-1)^{i+1}V_{i,n+1}(x)\\&\stackrel{\eqref{eq311}}{=}\frac{1-2\delta}{3}-\delta_6(j)\\&\stackrel{\ep_0<\delta/3}{<}\frac{1-\delta}{3}-\ep_0.\end{aligned}\end{equation}
Also, we have that
\begin{equation}\label{eq312}S_{k,n}(x)\stackrel{\eqref{eq311}}{=}(1-t(k))u(j)+t(k)\frac{(V_{1,n}+V_{2,n})(x)}{2}\stackrel{\eqref{p5}}{\le}(1-t(k))u(j)+\Delta,\end{equation}
and so,
\begin{equation}\label{eqb}\begin{aligned}\Big(\frac{2-2\delta}{3\|Z_{k,n}\|}S_{k,n}-(-1)^{i+1}&V_{i,n+1}\Big)(x)\stackrel{\eqref{eq309}}{\le} S_{k,n}(x)-(-1)^{i+1}V_{i,n+1}(x)\\&\stackrel{\eqref{eq311}\eqref{eq312}}{\le}(1-t(k))u(j)+\Delta-\frac{1-2\delta}{3}+\delta_6(j)\\&\stackrel{u(j)\le1}{\le} \frac{1}{2}+\frac{k}{M}-\frac{1-2\delta}{3}+2\Delta\\&\stackrel{\Delta<\delta/6}{<}\frac{1}{6}+\frac{8}{M}+\delta\le\frac{1-\delta}{3}-\ep_0.\end{aligned}\end{equation}
Notice that $u(j)\le1$ follows from Fact \ref{factou1} and $\Delta<\delta/6$ is one of the inequalities in \eqref{con3}.

Inequalities \eqref{eqa} and \eqref{eqb} yield
\begin{equation}\label{eq313}\Big|\Big((-1)^{i+1}V_{i,n+1}-\frac{2-2\delta}{3\|Z_{k,n}\|}S_{k,n}\Big)(x)\Big|<\frac{1-\delta}{3}-\ep_0,\end{equation}
and therefore we finish this case using the triangle inequality
$$\begin{aligned}
    \Big|\Big((-1)^{i+1}V_{i,n+1}-\frac{2-2\delta}{3\|Z_{k,n}\|}Z_{k,n}\Big)(x)\Big|&\le \Big|\Big((-1)^{i+1}V_{i,n+1}-\frac{2-2\delta}{3\|Z_{k,n}\|}S_{k,n}\Big)(x)\Big|\\&\;\;\;+\frac{2-2\delta}{3\|Z_{k,n}\|}\|S_{k,n}(x)-Z_{k,n}(x)\|\\&\stackrel{\eqref{eq309}\eqref{eq313}}{<}\frac{1-\delta}{3}.
\end{aligned}$$

\textbf{Case 6,} $x=(2(n+1)+i-1,j)$ where $j\in\N$. In this case,
\begin{equation}\label{eq314}\begin{aligned}V_{i,n+1}(x)=(-1)^{i-1}\frac{1-\delta}{3}e_1(j)\;\;\;\;,&\;\;\;\;V_{1,n}(x)=V_{2,n}(x)=0,\\U_{\frac{3+(-1)^n}{2}}(x)=\frac{1}{10}e_1(j)\;\;\;\;,&\;\;\;\;S_{k,n}(x)=(1-t(k))\frac{1}{10}e_1(j).\end{aligned}\end{equation}
First, we compute
\begin{equation}\label{eq315}
    \frac{2-2\delta}{3\|Z_{k,n}\|}\cdot\frac{1}{10}(1-t(k))\stackrel{\eqref{eq309}}{<}\frac{1}{10}\stackrel{\eqref{con0}}{<}\frac{1-\delta}{3}.
\end{equation}
Now, taking into account that $1-t(k)\ge\frac{1}{2}$ we have
$$\begin{aligned}\Big|\Big((-1)^{i+1}V_{i,n+1}-\frac{2-2\delta}{3\|Z_{k,n}\|}S_{k,n}\Big)(x)\Big|&\stackrel{\eqref{eq314}}{=}\Big|\Big(\frac{1-\delta}{3}-\frac{2-2\delta}{3\|Z_{k,n}\|}\cdot\frac{1}{10}(1-t(k))\Big)e_1(j)\Big|\\&\stackrel{\eqref{eq315}}{=}\Big(\frac{1-\delta}{3}-\frac{2-2\delta}{3\|Z_{k,n}\|}\cdot\frac{1}{10}(1-t(k))\Big)|e_1(j)|\\&\stackrel{\eqref{help}}{\le}\frac{1-\delta}{3}-\frac{1}{40}\stackrel{\eqref{con4}}{<}\frac{1-\delta}{3}-\ep_0.\end{aligned}$$
Therefore, by the triangle inequality and the last shown equation we get
$$\begin{aligned}
    \Big|\Big((-1)^{i+1}V_{i,n+1}-\frac{2-2\delta}{3\|Z_{k,n}\|}Z_{k,n}\Big)(x)\Big|&\le \Big|\Big((-1)^{i+1}V_{i,n+1}-\frac{2-2\delta}{3\|Z_{k,n}\|}S_{k,n}\Big)(x)\Big|+\ep_0\\&<\frac{1-\delta}{3}.
\end{aligned}$$

\textbf{Case 7,} $x=(p,j)\in\N^2$ different from the previous cases. In this case, $p\ge 2$ and hence we know that $V_{i,n}(x)\in\big\{(-1)^{i-1}\frac{1-\delta}{3}e_1(j),\frac{1}{10}e_{2n+i-1}(j),0\big\}$ where it cannot happen that $V_{1,n}(x)=V_{2,n}(x)=(-1)^{i-1}\frac{1-\delta}{3}e_1(j)$. Then,
\begin{equation}\label{eq316}\Big|\frac{V_{1,n}(x)+V_{2,n}(x)}{2}\Big|\le\frac{\frac{1-\delta}{3}+\frac{1}{10}}{2}=\frac{13-10\delta}{60}.\end{equation}
Also, since $x$ is different from that of Case 6, necessarily $p\neq 2(n+1)+i-1$ so that $V_{i,n+1}(x)\in\big\{0,\frac{1}{10}e_{2(n+1)+i-1}(j)\big\}$. Hence,
\begin{equation}\label{eq317}|V_{i,n+1}|\le\frac{1}{10}\;\;\text{ and }\;\;U_{\frac{3+(-1)^n}{2}}(x)=\frac{1}{10}e_1(j).\end{equation}
This, together with the fact that $\frac{1}{10}\stackrel{\eqref{con0}}{<}\frac{13-10\delta}{60}$ yields
\begin{equation}\label{eq318}
    \begin{aligned}
        |S_{k,n}(x)|&\le(1-t(k))\big|U_{\frac{3+(-1)^n}{2}}(x)\big|+t(k)\Big|\frac{V_{1,n}(x)+V_{2,n}(x)}{2}\Big| \\&\stackrel{\eqref{eq316}\eqref{eq317}}{\le}(1-t(k))\frac{1}{10}+t(k)\frac{13-10\delta}{60}<\frac{13-10\delta}{60}.
    \end{aligned}
\end{equation}
Therefore, taking into account that by \eqref{con4}, $\ep_0\le\frac{1-10\delta}{60}$ we finally conclude
$$\begin{aligned}
    \Big|\Big((-1)^{i+1}V_{i,n+1}-\frac{2-2\delta}{3\|Z_{k,n}\|}Z_{k,n}\Big)(x)\Big|&\stackrel{\eqref{eq309}}{\le} |V_{i,n+1}(x)|+|S_{k,n}(x)|+\ep_0\\&\stackrel{\eqref{eq317}\eqref{eq318}}{<}\frac{1}{10}+\frac{13-10\delta}{60}+\ep_0\le\frac{1-\delta}{3}.
\end{aligned}$$
\end{proof}

Our next objective is to prove that $(-1)^{i+1}V^i_{n+1}$ is the unique vector (up to $\ep$) satisfying the thesis of Lemma \ref{prelim2} (up to $\ep$). The proof relies on the $\ell_1$-like behaviour of our previously defined vectors.

\begin{lemma}\label{prelim3}
    There exists $C>0$ such that if $Z_{1,n}, \dots,Z_{8,n}$ are like in Lemma \ref{prelim2} for some $n\in\N$ then whenever there is $B\in X$ satisfying
    $$\|B\|,\Big\|B-\frac{2-2\delta}{3\|Z_{k,n}\|}Z_{k,n}\Big\|\le\frac{1-\delta}{3}+\ep,\;\;\;\;\forall k\neq i_0+1,$$
    for some $i_0\in\{1,2\}$ and $\ep>0$ it follows that
    $$\|B-(-1)^{i_0+1}V_{i_0,n+1}\|\le C\ep.$$
\end{lemma}
\begin{proof}We may assume without loss of generality that
$$B\in\text{span}\Big(\{V_{i,n}\}_{\substack{i=1,2\\n\in\N}}\cup\{U_1,U_2\}\Big).$$
    Let us consider $y_k=(1,(n+1)8+k)$ for $k\in\{1,\dots,8\}$ and $\Gamma=\{y_k\}_{\substack{k=1,\dots,8\\k\neq i_0+1}}$. Now, we consider the vectors $w, c,b,u_1,u_2,v_{i,m}\in\ell_\infty(\Gamma)\equiv\ell_\infty^7$ for $i=1,2$ and $m\in\N$ given by
    $$w(y_k)=\frac{1-2\delta}{3}-\delta_6(k)\;\;,\;\;c(y_k)=1\;\;,\;\;b(y_k)=B(y_k),$$
    $$u_1(y_k)=U_1(y_k)\;\;,\;\;u_2(y_k)=U_2(y_k)\;\;,\;\;v_{i,m}(y_k)=V_{i,m}(y_k).$$
    \textbf{Claim.} If $i_1\in\{1,2\}\setminus\{i_0g\}$ then the vectors $c,v_{i_1,n+1},v_{1,n},v_{2,n},u_1,u_2,w$ are linearly independent vectors from $\ell_\infty(\Gamma)$.
    \begin{proof}[Proof of the Claim]
    Let us first define the vectors $p,q\in\ell_\infty(\Gamma)$ as
    $$p(y_k)=g(x_k)\;\;\;,\;\;\;q(y_k)=f(x_k)\;\;\;\;\;\forall k\in\{1,\dots,8\}\setminus\{i_0+1\}.$$
    Consider also $e_s\in\ell_\infty(\Gamma)$ for $s\in\N$ as
    $$e_s(y_k)=\begin{cases}1\;\;\;&\text{if }k=s\\0&\text{if }K\neq s.\end{cases}$$
    If one takes into account the immediate fact that $p$, $q$ and $c$ are respectively the evaluation in $\{x_k\}_{\substack{k=1,\dots,8\\k\neq i_0+1}}$ of a polynomial of degree exactly 2, a polynomial of degree exactly 1 and a non-zero constant polynomial, it follows that the vectors from  $\beta=\{p,q,c,e_{i_1+1},e_6,e_7,e_8\}$ are linearly independent. Hence, $\beta$ is a basis for $\ell_\infty(\Gamma)$ and we just need to compute the determinant of the matrix $M$ whose rows are the coordinates of the vectors $c,v_{i_1,n+1},v_{1,n},v_{2,n},u_1,u_2,w$ for the basis $\beta$. It is easy to check that
    $$M=\begin{pmatrix} 
    0 & 0 & 1 & 0 & 0 & 0 & 0\\
    0 & 0 & (-1)^{i_1-1}\frac{1-\delta}{3} & (-1)^{i_1}\Delta & 0 & 0 & 0\\
    1 & 0 & -\frac{2+2\delta}{3} & 0 & 0 & -\Delta & 0\\
    1 & 0 & -\frac{2+2\delta}{3} & 0 & 0 & \Delta & 0\\
    0 & 1 & 0 & 0 & 0 & 0 & 0\\
    0 & 0 & 0 & 0 & 0 & 0 & \Delta\\
    0 & 0 & \frac{1-2\delta}{3} & 0 & -\Delta & 0 & 0
    \end{pmatrix}.$$
    To show that $det(M)\neq0$ it suffices to repeatedly use the Laplace expansion of the determinant.
    \end{proof}
    Let us continue now the proof of Lemma \ref{prelim3}. We consider $i_1\in\{1,2\}$ distinct from $i_0$. There must exist unique $\lambda,\lambda_1,\dots,\lambda_5,\rho_{i,m}\in\R$ for $i=1,2$ and $m\in\N\setminus\{n,n+1\}$ such that just finitely many of them are nonzero and satisfy that
    \begin{equation}\label{eq8}B=\lambda(-1)^{i_0+1}V_{i_0,n+1}+\lambda_1V_{i_1,n+1}+\lambda_{2}V_{1,n}+\lambda_3V_{2,n}+\lambda_4U_1+\lambda_5 U_2+\sum\limits_{\substack{i=1,2\\m\in\N\setminus\{n,n+1\}}} \rho_{i,m}V_{i,m}.\end{equation}
    If $m\in\N\setminus\{n,n+1\}$ and $i=1,2$ then
    $$v_{i,m}(y_k)=\begin{cases}(-1)^{i-1}\big(\frac{1-2\delta}{3}-\delta_6(k)\big)=(-1)^{i-1}w(y_k)\;&\text{ if }m\neq n\text{ mod }2,\\0&\text{ if }m=n\text{ mod }2.\end{cases}$$
    Therefore, since $\frac{1-\delta}{3}c=(-1)^{i_0-1}v_{i_0,n+1}$ we have that
    \begin{equation}\label{eq9}b=\lambda\frac{1-\delta}{3}c+\lambda_1v_{i_1,n+1}+\lambda_{2}v_{1,n}+\lambda_3v_{2,n}+\lambda_4u_1+\lambda_5 u_2+\lambda_6w,\end{equation}
    where $\lambda_6=\sum\limits_{\substack{m\in\N\setminus\{n,n+1\}\\m\neq n\text{ mod }2}}\rho_{1,m}-\rho_{2,m}$.
    Also, for $i=1,2$,
    $$v_{i,n}(y_k)=v(k)-\frac{2+2\delta}{3}+(-1)^{i}\delta_7(k)\;\;,\;\;v_{i,n+1}(y_k)=(-1)^{i-1}\Big(\frac{1-\delta}{3}-\delta_{i+1}(k)\Big),$$
    $$u_{\frac{3+(-1)^n}{2}}(y_k)=u(k)\;\;\;,\;\;\;u_{\frac{3-(-1)^n}{2}}(y_k)=\delta_8(k).$$
    From the previous Claim we know that the vectors $c,v_{i_1,n+1},v_{1,n},v_{2,n},u_1,u_2,w$ are linearly independent vectors from $\ell_\infty^7$. Hence, we may use Lemma \ref{prelim1} with $a=\frac{3}{1-\delta}b$ which provides us with some $K
    >0$ independent of $\ep$ and $n$ such that $|\lambda-1|,\big|\lambda_s\frac{3}{1-\delta}\big|\le K\|a-c\|$ for $s=1,\dots,6$. Thus,
    \begin{equation}\label{eq10}
        |\lambda-1|,|\lambda_s|\le 4K\Big\|b-\frac{1-\delta}{3}c\Big\|\;\;\;\;\text{ for }\;\;\;s=1,\dots,5.
    \end{equation}
    From the hypothesis we deduce that for $k\neq i_0+1$,
    $B(y_k)\le\frac{1-\delta}{3}+\ep$ and 
    $$\frac{2-2\delta}{3}-B(y_k)\stackrel{\eqref{Zeq}}{=}\frac{2-2\delta}{3\|Z_{k,n}\|}Z_{k,n}(y_k)-B(y_k)\le\frac{1-\delta}{3}+\ep.$$
    This implies that $|B(y_k)-\frac{1-\delta}{3}|\le\ep$ for $k\in\{1,\dots,8\}\setminus\{i_0+1\}$ and hence
    \begin{equation}\label{eq7}\Big\|b-\frac{1-\delta}{3}c\Big\|\le\ep.
    \end{equation}
    Taking into account inequalities \eqref{eq10} and \eqref{eq7} we deduce that
    \begin{equation}\label{eq11}
        |\lambda-1|,|\lambda_s|\le 4K\ep\;\;\;\;\text{ for }\;\;\;s=1,\dots,5.
    \end{equation}
    Finally, we consider 
    $$\widetilde b=\restr{B}{\{2(n+1)+i_0-1\}\times\N}\;\;\;\;\text{ and }\;\;\;\;\widetilde d=\widetilde b-\sum\limits_{\substack{i=1,2\\m\in\N\setminus\{n,n+1\}}} \rho_{i,m}\restr{V_{i,m}}{\{2(n+1)+i_0-1\}\times\N}.$$
    Notice that for every $1\le i\le 2$ and $m\in\N\setminus\{n,n+1\}$ it holds that $\restr{V_{i,m}}{\{2(n+1)+i_0-1\}\times\N}=\frac{1}{10}e_{2m+i-1}$. Also, by \eqref{eq8} and \eqref{eq11} we have that
    $$\Big\|\widetilde d-\frac{1-\delta}{3}e_1\Big\|=\|\widetilde d-(-1)^{i_0-1}\restr{V_{i_0,n+1}}{\{2(n+1)+i_0-1\}\times\N}\|\le|\lambda-1|+\sum_{s=1}^5|\lambda_s|\stackrel{\eqref{eq11}}{\le}24K\ep.$$
    Hence, by the triangle inequality,
    $$\begin{aligned}\|\widetilde b\|&=\Bigg\|\widetilde d+\sum\limits_{\substack{i=1,2\\m\in\N\setminus\{n,n+1\}}} \rho_{i,m}\frac{1}{10}e_{2m+i-1}\Bigg\|\\&=\Bigg\|\Big(\widetilde d-\frac{1-\delta}{3}e_1\Big)+\Bigg(\frac{1-\delta}{3}e_1+\frac{1}{10}\cdot\sum\limits_{\substack{i=1,2\\m\in\N\setminus\{n,n+1\}}} \rho_{i,m}e_{2m+i-1}\Bigg)\Bigg\|\\&\ge\Bigg(\frac{1-\delta}{3}+\frac{1}{10}\cdot\sum\limits_{\substack{i=1,2\\m\in\N\setminus\{n,n+1\}}} |\rho_{i,m}|\Bigg)-24K\ep.\end{aligned}$$
    On the other hand $\|\widetilde b\|\le\|B\|\le\frac{1-\delta}{3}+\ep$. Therefore,
    $$\frac{1}{10}\cdot\sum\limits_{\substack{i=1,2\\m\in\N\setminus\{n,n+1\}}} |\rho_{i,m}|\le(24K+1)\ep.$$
    This finishes the proof since
    $$\begin{aligned}&\|B-(-1)^{i_0+1}V_{i_0,n+1}\|\\&\stackrel{\eqref{eq8}}{=}\Bigg\|(\lambda-1)(-1)^{i_0+1}V_{i_0,n+1}+\lambda_1V_{i_1,n+1}+\lambda_{2}V_{1,n}+\lambda_3V_{2,n}+\lambda_4U_1+\lambda_5 U_2+\sum\limits_{\substack{i=1,2\\m\in\N\setminus\{n,n+1\}}} \rho_{i,m}V_{i,m}\Bigg\|\\&\le|\lambda-1|+\sum_{s=1}^5|\lambda_s|+\sum\limits_{\substack{i=1,2\\m\in\N\setminus\{n,n+1\}}} |\rho_{i,m}|\stackrel{\eqref{eq11}}{\le}(264K+10)\ep.\end{aligned}$$
\end{proof}
We are finally ready to prove Theorem \ref{mainTH}, which we state here in a more precise way.
\begin{theorem}\label{precth}
    There are $\ep_0,C>0$ such that if $D$ is a convex subset of $X$ with $0\in D$ and $R:X\to X$ is a $(1+\ep_0/C)$-Lipschitz retraction onto $D$ satisfying for $i,j=1,2$,
    $$\|R((-1)^jV_{i,1})-(-1)^jV_{i,1}\|,\|R((-1)^jU_i)-(-1)^jU_i\|\le \ep_0,$$
    then for every $i,j=1,2$ and $n\in\N$,
    $$\|R((-1)^jV_{i,n})-(-1)^jV_{i,n}\|\le\ep_0.$$
\end{theorem}
\begin{proof}
    We will consider $\ep_0,C>0$ given respectively by Lemma \ref{prelim2} and Lemma \ref{prelim3}. Let us proceed by induction, showing that for every $n\in\N$ we have that $\|R((-1)^jV_{i,n})-(-1)^jV_{i,n}\|\le\ep_0$. The case $n=1$ follows directly from the assumptions above. Now, if we assume that $\|R((-1)^jV_{i,n})-(-1)^jV_{i,n}\|\le\ep_0$ for some $n\in\N$, let us prove that
    \begin{equation}\label{hypoeq}\|R((-1)^jV_{i,n+1})-(-1)^jV_{i,n+1}\|\le\ep_0.\end{equation}
    We first denote for $i,j=1,2$ and $k=1,\dots,8$,
    $$W_i^j:=(-1)^jR((-1)^jV_{i,n})\;\;\;,\;\;\; E^j_i:=(-1)^jR((-1)^jU_i)\;\;\;,\;\;\;B^j_i=(-1)^jR((-1)^{i+j+1}V_{i,n+1}),$$
    $$Z_{k,n}^j:=(1-t(k))E^j_{\frac{3+(-1)^n}{2}}+t(k)\frac{W_1^j+W_2^j}{2}.$$
    By property \eqref{p2} from Proposition \ref{c1},
    \begin{equation}\label{eqf1}\|B^j_i\|=\|R((-1)^{i+j+1}V_{i,n+1})-R(0)\|\le(1+\ep_0/C)\frac{1-\delta}{3}.\end{equation}
    It is immediate that $(-1)^jE_i^j,(-1)^jW_i^j\in D$ and hence $\frac{2-2\delta}{3\|Z_{k,n}^j\|}(-1)^jZ_{k,n}^j\in D$ for $j\in\{1,2\}$. Notice also that by the induction hypothesis we know that $E_1^j,E_2^j,W_1^j,W_2^j$ satisfy the requirements of Lemma \ref{prelim2}. Therefore,
    by Lemma \ref{prelim2}, if $k\neq i+1$,
    \begin{equation}\label{eqf2}\begin{aligned}\Big\|B^j_i-\frac{2-2\delta}{3\|Z^j_{k,n}\|}Z^j_{k,n}\Big\|=&\Big\|R((-1)^{i+j+1}V_{i,n+1})-R\Big(\frac{2-2\delta}{3\|Z_{k,n}^j\|}(-1)^jZ_{k,n}^j\Big)\Big\|\\\le&(1+\ep_0/C)\Big\|(-1)^{i+1}V_{i,n+1}-\frac{2-2\delta}{3\|Z^j_{k,n}\|}Z^j_{k,n}\Big\|\\=&(1+\ep_0/C)\frac{1-\delta}{3}.\end{aligned}\end{equation}
    Putting together inequalities \eqref{eqf1} and \eqref{eqf2} we have that whenever $k\neq i+1$,
    $$\|B^j_i\|,\Big\|B^j_i-\frac{2-2\delta}{3\|Z^j_{k,n}\|}Z^j_{k,n}\Big\|\le\frac{1-\delta}{3}+\ep_0/C.$$
    Clearly, from the induction hypothesis we have $\|W_i^j-V_{i,n}\|\le\ep_0$ and $\|E_i^j-U_i\|\le\ep_0$ for $i,j\in\{1,2\}$. Therefore, we are allowed to use Lemma \ref{prelim3} in this situation. Then, for every $1\le i,j\le2$,
    $$\|B^j_i-(-1)^{i+1}V_{i,n+1}\|\le\ep_0.$$
    This proves \eqref{hypoeq} and hence finishes the induction  since for every $1\le i,j\le2$,    $$\begin{aligned}\|R((-1)^{i+j+1}V_{i,n+1})-(-1)^{i+j+1}V_{i,n+1}\|=&\|(-1)^jR((-1)^{i+j+1}V_{i,n+1})-(-1)^{i+1}V_{i,n+1}\|\\=&\|B^j_i-(-1)^{i+1}V_{i,n+1}\|\le\ep_0,\end{aligned}$$
    
\end{proof}

\begin{proof}[Proof of Theorem \ref{mainTH}]
    We take $\lambda=1+\ep_0/C>1$. Assume that there is a $\lambda$-Lipschitz retraction $R$ from $X$ onto a generating compact and convex subset $K$ of $X$. Then, for every $k\in\N$ we may shift $R$ by an element $x_0\in X$ and dilate it with ratio $k$ to obtain
    $$R_k(x)=kR\Big(\frac{x}{k}+x_0\Big)-kx_0.$$
    Notice that $R_k$ is a $\lambda$-Lipschitz retraction onto $k(K-x_0)$. Picking $x_0\in K$ appropriately we get from the fact that $K$ is generating that $R_k(x)\xrightarrow[k\to\infty]{\|\cdot\|}x$.
    Therefore, for large enough $k\in\N$, $R_k$ meets the hypothesis of Theorem \ref{precth} and hence
    $$\|R_k(V_{1,n})-V_{1,n}\|\le\ep_0\;\;\;\;\forall n\in\N.$$
    Clearly, $\|V_{1,m}-V_{1,n}\|\ge1/5$ for every $n,m\in\N$ distinct. Hence, for every $m,n\in\N$ distinct, $\|R_k(V_{1,m})-R_k(V_{1,n})\|\ge 1/5-2\ep_0>0.$
    Therefore we find the contradiction since the compact set $k(K-x_0)$ contains the sequence $\big(R_k(V_{1,n})\big)_{n\in\N}$ which has no Cauchy subsequence.
\end{proof}

It is worth mentioning that $X$ does enjoy both linear and nonlinear approximation properties. In fact, we have the following straightforward result.

\begin{proposition}
    The space 
    $$X=\overline{\text{span}}\Big(\{V_{i,n}\}_{\substack{i=1,2\\n\in\N}}\cup\{U_1,U_2\}\Big)$$
    is isomorphic to $\ell_1$.
\end{proposition}
\begin{proof}
    We know that $X=\text{span}(\{V_{1,1},V_{2,1},U_1,U_2\})\oplus\overline{\text{span}}\Big(\{V_{i,n}\}_{\substack{i=1,2\\n\ge2}}\Big)$. Therefore, it is enough to show that $Y=\overline{\text{span}}\Big(\{V_{i,n}\}_{\substack{i=1,2\\n\ge2}}\Big)$ is isomorphic to $\ell_1$. Clearly $\restr{V_{i,n}}{\{2\}\times\N}=\frac{1}{10}e_{2n+i-1}$ for every $1\le i\le2,$ $n\ge2$ and so for every sequence $(\lambda_{i,n})_{\substack{i=1,2\\n\ge2}}\subset\R$ with finitely many nonzero elements it holds that
    $$\begin{aligned}\bigg\|\sum\limits_{\substack{i=1,2\\n\ge2}}\lambda_{i,n}V_{i,n}\bigg\|\ge&\bigg\|\sum\limits_{\substack{i=1,2\\n\ge2}}\lambda_{i,n}\restr{V_{i,n}}{\{2\}\times\N}\bigg\|=\bigg\|\sum\limits_{\substack{i=1,2\\n\ge2}}\lambda_{i,n}\frac{1}{10}e_{2n+i-1}\bigg\|=\frac{1}{10}\sum\limits_{\substack{i=1,2\\n\ge2}}|\lambda_{i,n}|.\end{aligned}$$
\end{proof}

\begin{corollary}\label{corl1}
There is a Lipschitz retraction from $X=\overline{\text{span}}\Big(\{V_{i,n}\}_{\substack{i=1,2\\n\in\N}}\cup\{U_1,U_2\}\Big)$
onto a generating compact and convex subset of $X$.
\end{corollary}
\begin{proof}
    Since $\ell_1$ has a Schauder basis, necessarily $X$ has a basis, too. Then, the statement follows from Theorem 3.3 in \cite{HM21}.\end{proof}

\textbf{Data availability statement.} My manuscript has no associated data.

\textbf{Conflict of interest statement.} There is no conflict of interest.

\textbf{Founding information.} This work has been supported by PID2021-122126NB-C31 AEI (Spain) project, by FPU19/04085 MIU (Spain) Grant, by Junta de Andalucia Grants FQM-0185 and PY20\_00255, by GA23-04776S project (Czech Republic) and by SGS22/053/OHK3/1T/13 project (Czech Republic).


\bigskip

\begin{thebibliography}{99}

 \bibitem {BL2000} Yoav Benyamini and Joram Lindenstrauss, \textit{Geometric nonlinear functional analysis}, Vol.1 of American Mathematical Society colloquium publications 48-. American Mathematical Society, 2000

 \bibitem {Enf73b} Per Enflo. \textit{A Banach space with basis constant >1} 1. Ark. Mat., 11:103–107, 1973.

 \bibitem {Enf73} Per Enflo. \textit{A counterexample to the approximation problem in Banach spaces.} Acta
Math., 130:309–317, 1973.

\bibitem {Fab1} Marián Fabian, Petr Habala, Petr Hájek, Vicente Montesinos, and Václav Zizler. \textit{Banach space theory.} CMS Books in Mathematics/Ouvrages de Mathématiques de la SMC.

\bibitem {GK03} Gilles Godefroy and Nigel J. Kalton. \textit{Lipschitz-free Banach spaces}. Volume 159, pages 121–141. 2003. Dedicated to Professor Aleksander Pe\l czy\'{n}ski on the occasion of his 70th birthday.

\bibitem {GP19} Luis C. García-Lirola and Antonín Procházka. \textit{Pe\l czy\'{n}ski space is isomorphic to the Lipschitz free space over a compact set.} Proc. Amer. Math. Soc., 147(7):3057–3060, 2019.

\bibitem {GMZ16} Antonio J. Guirao, Vicente Montesinos, and Václav Zizler. \textit{Open problems in the geometry and analysis of Banach spaces.} Springer, [Cham], 2016.

\bibitem {GO14} Gilles Godefroy and Narutaka Ozawa. \textit{Free Banach spaces and the approximation properties.} Proc. Amer. Math. Soc., 142(5):1681–1687, 2014.

\bibitem {God215} Gilles Godefroy. \textit{Extensions of Lipschitz functions and Grothendieck’s bounded approximation property.} North-West. Eur. J. Math., 1:1–6, 2015.

\bibitem {God15} Gilles Godefroy. \textit{A survey on Lipschitz-free Banach spaces.} Comment. Math., 55(2):89–118, 2015.

\bibitem {God20} Gilles Godefroy. \textit{Lipschitz approximable Banach spaces.} Comment. Math. Univ. Carolin., 61(2):187–193, 2020.

\bibitem {HM21} Petr Hájek and Rubén Medina. \textit{Compact retractions and Schauder decompositions in Banach spaces}. Trans. Amer. Math. Soc., 376(2):1343–1372, 2023.

\bibitem {HM22} Petr Hájek and Rubén Medina. \textit{Retractions and the bounded approximation property in Banach spaces.} Mediterr. J. Math., 20(2):Paper No. 75, 13, 2023.

\bibitem {Kal12} Nigel J. Kalton. \textit{The uniform structure of Banach spaces.} Math. Ann., 354(4):1247–1288, 2012.


\bibitem {Med23} Rubén Medina. \textit{Compact Hölder retractions and nearest point maps}. Adv. Math., 428:Paper No. 109140, 13, 2023.






\end{thebibliography}
\end{document}